\documentclass[a4paper,12pt]{amsart}
\usepackage{vmargin}
\usepackage[pdftex]{graphicx}
\usepackage{amsfonts}
\usepackage{amssymb}
\usepackage{mathrsfs}
\usepackage{stmaryrd}
\usepackage{amsmath}
\usepackage{amsthm}
\usepackage{enumerate}
\usepackage{nomencl}
\usepackage[bookmarks=true]{hyperref}   
\usepackage{esint}
\usepackage{dsfont}
\usepackage{upgreek}
\usepackage{color}
\usepackage{comment}
\usepackage{float}

\makeatletter
\renewcommand\section{\@startsection{section}{1}{\z@}%
                       {-3\p@ \@plus -4\p@ \@minus -4\p@}%
                       {3\p@ \@plus 4\p@ \@minus 4\p@}%
                      {\normalfont\normalsize\centering\scshape}}
\makeatother

\hypersetup{
    colorlinks,%
}

\author{Lashi Bandara}
\title[Cont. of solns. to space-varying pointwise lin. ell. eqns.] {Continuity of solutions to space-varying pointwise linear elliptic equations}
\date{\today}
\address{Lashi Bandara, Mathematical Sciences,
Chalmers University of Technology and University of Gothenburg, SE-412 96, Gothenburg, Sweden}
\urladdr{\href{http://www.math.chalmers.se/~lashitha}{http://www.math.chalmers.se/~lashitha}}
\email{\href{mailto:lashi.bandara@chalmers.se}{lashi.bandara@chalmers.se}}

\def\colour{\colour}

\def\colour{\color}

\newtheorem{theorem}{Theorem}[section]
\newtheorem{corollary}[theorem]{Corollary}
\newtheorem{lemma}[theorem]{Lemma}
\newtheorem{proposition}[theorem]{Proposition}

\newtheorem{remark}[theorem]{Remark}

\newcommand{\mdot}{\cdotp}

\newcommand{\cbrac}[1]{\left(#1\right)}

\newcommand{\dbrac}[1]{\left\{#1\right\}}
\newcommand{\modulus}[1]{|#1|}

\newcommand{\set}[1]{\dbrac{#1}}

\newcommand{\dom}{ {\mathcal{D}}}
\newcommand{\ran}{ {\mathcal{R}}}
\newcommand{\nul}{ {\mathcal{N}}}

\newcommand{\R}{\mathbb{R}}
\newcommand{\C}{\mathbb{C}}

\newcommand{\script}[1]{\mathscr{#1}}

\DeclareMathOperator{\re}{Re}			

\renewcommand{\emptyset}{\varnothing}
\newcommand{\union}{\cup}

\newcommand{\intersect}{\cap}

\newcommand{\rest}[1]{{{\lvert_{}}_{}}_{#1}}
\newcommand{\close}[1]{\overline{#1}}		

\renewcommand{\epsilon}{\varepsilon}

\renewcommand{\phi}{\varphi}

\newcommand{\norm}[1]{\| #1 \|}			

\DeclareMathOperator{\divv}{div}		

\newcommand{\Ric}{{\rm Ric}}			

\DeclareMathOperator{\inj}{inj} 		

\newcommand{\Tensors}[1][{}]{{\mathcal{T}}^{(#1)}}	

\DeclareMathOperator{\proj}{\mathbf{P}}	
\newcommand{\tanb}{{\rm T}}		
\newcommand{\cotanb}{{\rm T}^\ast}	

\newcommand{\pushf}[1]{{#1}_\ast}			
\newcommand{\pullb}[1]{{#1}^\ast}			

\DeclareFontFamily{OT1}{restrictfont}{}
\DeclareFontShape{OT1}{restrictfont}{m}{n}{<-> fmvr8x}{}

\newcommand{\adj}[1]{{#1}^\ast}			
\newcommand{\extd}{{\rm d}}			

\newcommand{\inprod}[1]{\langle #1 \rangle}	
\newcommand{\grad}{\nabla}			
\newcommand{\conn}[1][{}]{{\grad_{{#1}}}}		
\newcommand{\Leb}[1][{}]{\script{L}^{#1}}			

\newcommand{\spec}{\sigma}		
\newcommand{\rset}{\rho}				

\newcommand{\Lp}[2][{}]{{\rm L}^{#2}_{\rm #1}}		
\newcommand{\Ck}[2][{}]{{\rm C}^{#2}_{\rm #1}}		
\newcommand{\Sob}[2][{}]{{\rm W}^{#2}_{\rm #1}}		

\newcommand{\Sec}[1]{\mathrm{S}_{#1}}
\newcommand{\OSec}[1]{\mathrm{S}^\mathrm{o}_{#1}}

\newcommand{\iden}{{\mathrm{I}}}
\DeclareMathOperator{\sgn}{sgn}

\newcommand{\Hil}{\script{H}}			
\newcommand{\Lap}{\Delta}			

\newcommand{\sE}{\script{E}}

\newcommand{\cM}{\mathcal{M}} 
\newcommand{\cN}{\mathcal{N}}

\newcommand{\cS}{\mathcal{S}}

\newcommand{\A}{A}

\newcommand{\Hol}{{\rm Hol}}

\newcommand{\mg}{\mathrm{g}}
\newcommand{\mgt}{{\tilde{\mg}}}

\newcommand{\Poincare}{Poincar\'e~}		

\newcommand{\Div}{\mathrm{L}}

\newcommand{\B}{\mathrm{B}}

\newcommand{\RCD}{\mathrm{RCD}}

\newcommand{\hk}{\uprho}

\newcommand{\RNum}[1]{\uppercase\expandafter{\romannumeral #1\relax}}

\begin{document}

\maketitle

\begin{abstract}
We consider pointwise linear elliptic 
equations of the form $\Div_x u_x = \eta_x$
on a smooth compact manifold
where the operators $\Div_x$ are in divergence
form with real, bounded, measurable
coefficients that vary in the space variable $x$. 
We establish $\Lp{2}$-continuity
of the solutions at $x$ whenever
the coefficients of $\Div_x$ are 
$\Lp{\infty}$-continuous at $x$ 
and the initial datum is $\Lp{2}$-continuous
at $x$. This is obtained
by reducing the continuity of solutions 
to a homogeneous Kato square root 
problem. 
As an application,
we consider a time evolving
family of metrics $\mg_t$
that is tangential to 
the Ricci flow almost-everywhere along
geodesics when 
starting with a smooth 
initial metric.
Under the assumption that our
initial metric is a rough metric on $\cM$ with a
$\Ck{1}$ heat kernel on a 
``non-singular'' nonempty open subset $\cN$,
we show that $x \mapsto \mg_t(x)$
is continuous whenever $x \in \cN$.
\end{abstract}
\vspace*{-0.5em}
\tableofcontents
\vspace*{-2em}

\parindent0cm
\setlength{\parskip}{\baselineskip}

\section{Introduction}

The object of this paper is to consider the continuity
of solutions to certain linear elliptic partial 
differential equations, where the differential operators
themselves vary from point to point.
To fix our setting, let $\cM$ be a smooth compact Riemannian
manifold, and $\mg$ a smooth metric.
Near some point $x_0 \in \cM$, we fix
an open set $U_0$ containing $x_0$.
We assume that  
$U_0 \ni x\mapsto \Div_x$,
are space-varying elliptic, second-order divergence
form operators
with real, bounded, measurable coefficients.
The equation at the centre of
our study is the following \emph{pointwise} linear 
problem
\begin{equation*}
\tag{PE}
\label{Def:PE}  
\Div_x u_x = \eta_x
\end{equation*}
for suitable source data $\eta_x \in \Lp{2}(\cM)$.
Our goal is to establish the continuity of
solutions $x \mapsto u_x$ (in $\Lp{2}(\cM)$)
under sufficiently general 
hypotheses on $x \mapsto \Div_x$ and
$x \mapsto \eta_x$.
 
There are abundant equations of the form \eqref{Def:PE} 
that arise naturally.
An important and large class of such equations
arise as \emph{continuity equations}. 
These equations are typically of the form
\begin{equation*}
\label{Def:CE}
\tag{CE}
-\divv_{\mg,y} f_{x}(y) \conn u_{x,v}(y) = \extd_x (f_{x}(y))(v),
\end{equation*}
where $\gamma:I \mapsto \cM$ is a
smooth curve, $\gamma(0) = x$ and $\dot{\gamma}(0) = v$, 
and where this equation holds in a suitable
weak sense in $y$.
These equations play an important role
in geometry, and more recently
in mass transport and the geometry
of measure metric spaces. 
See the book \cite{Villani} by Villani, the paper \cite{AT} by Ambrosio
and Trevisan, and references within.

The operators $\Div_x$ have 
the added complication 
that their domain  may 
vary as the point $x$ varies. 
That being said, 
a redeeming quality 
is that they facilitate a certain 
\emph{disintegration}. 
That is, considerations in $x$
(such as continuity and differentiability), 
can be obtained via weak solutions in $y$.
This structural feature facilitates attack by
techniques from operator theory and
harmonic analysis as we demonstrate in this paper.

A very particular instance of
the continuity equation
that has been a core motivation
is where, in the equation \eqref{Def:CE},
the term $f_x(y) = \hk^\mg_t(x,y)$, the
heat kernel associated to the Laplacian $\Lap_\mg$. 
In this situation, Gigli and Mantegazza in \cite{GM} 
define a metric tensor 
$\mg_t(x)(v,u) = \inprod{\Div_x u_{x,v}, u_{x,u}}$
for vectors $u, v \in \tanb_x \cM$.
The regularity of the metric is then regularity in $x$,
and for an initial smooth metric,
the aforementioned authors show that this
evolving family of metrics 
are smooth.
More interestingly, they demonstrate that 
$$\partial_t \mg_t(\dot{\gamma}(s),\dot{\gamma}(s))\rest{t = 0} = -2 \Ric_{\mg}(\dot{\gamma}(s),\dot{\gamma}(s)),$$ 
for almost-every $s$ along geodesics $\gamma$.
That is, this flow $\mg_t$ is 
\emph{tangential} to the Ricci flow 
almost-everywhere along geodesics.

In \cite{BLM},  
Bandara, Lakzian and Munn study
a generalisation of this flow 
by considering divergence form elliptic equations with 
bounded measurable coefficients.
They obtain regularity properties
for $\mg_t$ when the heat kernel is Lipschitz
and improves to a $\Ck{k}$ map ($k \geq 2$)
on some non-empty open set in the manifold.
Their study was motivated by attempting to 
describe the evolution of 
geometric conical singularities
as well as other singular spaces.
As an application
we return to this work 
and consider the case when $k = 1$. 

To describe the main theorem of
this paper, let us 
give an account of some useful
terminology. We assume that $\Div_x$ are
defined through a space-varying 
symmetric form $J_x[u,v] = \inprod{A_x \conn u, \conn v}$,
where each $A_x$ is a bounded, measurable, 
symmetric $(1,1)$ tensor field
which is elliptic at $x$: there exist
$\kappa_x > 0$ such that $J_x[u,u] \geq \kappa_x \norm{\conn u}^2$.
Next, let us be precise about the notion of $\Lp{p}$-continuity.
We say that $x\mapsto u_x$ is $\Lp{p}$-continuous 
if, given an $\epsilon > 0$,
there exists an open set $V_{x,\epsilon}$ containing 
$x$ such that, whenever $y \in V_{x,\epsilon}$,
we have that $\norm{u_y - u_x}_{\Lp{p}} < \epsilon$.
With this in mind, we showcase our main theorem.

\begin{theorem}
\label{Thm:Main}
Let $\cM$ be a smooth manifold and $\mg$ a smooth 
metric.
At $x \in \cM$ suppose that $x \mapsto A_x$ 
are real, symmetric, elliptic, bounded 
measurable coefficients that are $\Lp{\infty}$-continuous at $x$,
and that  $x \mapsto \eta_x$ is $\Lp{2}$-continuous at $x$.
If $x \mapsto u_x$ 
solves \eqref{Def:PE} at $x$, 
then $x \mapsto u_x$ is $\Lp{2}$-continuous at $x$.
\end{theorem}

As aforementioned, a complication that 
arises in proving this 
theorem is that domains $\dom(\Div_x)$ may vary
with $x$. However, since the solutions
$x \mapsto u_x$ live at the level of the
resolvent of $\Div_x$, there
is hope to reduce this problem
to the difference of its square root, 
which incidentally has the fixed
domain $\Sob{1,2}(\cM)$.
As a means to this end, we make connections between 
the study of the $\Lp{2}$-continuity of these 
solutions to solving a 
\emph{homogeneous Kato square root problem}.

Let $B$ be complex
and in general, non-symmetric coefficients
and let  $J_B[u,v] = \inprod{B \conn u, \conn v}$
whenever $u, v \in \Sob{1,2}(\cM)$. 
Suppose that there exists $\kappa > 0$
such that  $\re J_B[u,u] \geq \kappa \norm{\conn u}$. 
Then, the Lax-Milgram theorem 
yields a closed, densely-defined
operator $\Div_Bu = -\divv_\mg B \conn u$.
The homogeneous Kato square root problem is  to assert that 
$\dom(\sqrt{-\divv_\mg B \conn}) = \Sob{1,2}(\cM)$
with the estimate 
$\norm{\sqrt{-\divv_\mg B\conn} u} \simeq \norm{\conn u}$.

The Kato square root problem on $\R^n$ is
the case $\cM = \R^n$
and this  
conjecture resisted resolution
for almost forty years 
before it was finally settled 
in 2002 by 
Auscher, Hoffman, Lacey, McIntosh and Tchamitchian 
in \cite{AHLMcT}.
Later, this problem 
was rephrased from a first-order
point of view by Axelsson, Keith, and McIntosh
in \cite{AKMc}. This seminal paper contained
the first Kato square root result for compact
manifolds, but the operator
in consideration was inhomogeneous.

In the direction of non-compact manifolds, 
this approach was subsequently used by 
Morris in \cite{Morris} to solve a similar
inhomogeneous problem on Euclidean submanifolds. 
Later, in the intrinsic geometric setting, 
this problem was solved by McIntosh and the author
in \cite{BMc} on smooth manifolds (possibly non-compact) 
assuming 
a lower bound on injectivity radius and a bound on Ricci curvature.
Again, these results were for inhomogeneous operators
and are unsuitable for
our setting where we deal with the homogeneous kind. In \S\ref{Sec:Kato},
we use the framework and other results in \cite{BMc} to
solve the homogeneous problem.

The solution to the homogeneous Kato square
root problem is relevant to us for the following
reason. 
Underpinning the Kato square root estimate is 
a \emph{functional calculus} and due to the fact that we allow
for complex coefficients, we obtain holomorphic
dependency of this calculus.
This, in turn, provides us with 
Lipschitz estimates for small perturbations
of the (non-linear) operator $B \mapsto \sqrt{-\divv_\mg B \conn}$.
This is the crucial estimate that yields
the continuity result in our main theorem.

To demonstrate the usefulness of 
our results, we give an application of
Theorem \ref{Thm:Main}
to the aforementioned geometric flow introduced
by Gigli and Mantegazza.
In \S\ref{Sec:App}, we demonstrate under a very weak 
hypothesis that this flow is continuous.
We remark that this is the first 
instance known to us 
where the Kato square root problem
has been used in the context
of geometric flows. 
We hope that this paper provides an impetus
to further investigate 
the relevance of Kato square root results to geometry, particularly 
given the increasing prevalence  of the continuity
equation in geometric problems.
\section*{Acknowledgements}

This research was conducted during the
``Junior Trimester Program on Optimal Transport'' at
the Hausdorff Research Institute for Mathematics
in Bonn, Germany. We thank the Institute for funding and support.

The author thanks Sajjad Lakzian, Mike Munn and Rupert McCallum
for useful discussions that lead to this work.
Moreover, the author would like to acknowledge 
and thank 
Alan McIntosh for his continual encouragement
and support in fostering connections 
between harmonic analysis and geometry.
\section{The structure and solutions of the equation}

Throughout this paper, let us fix the manifold 
$\cM$ to be a smooth, compact manifold
and, unless otherwise stated, let $\mg$
be a smooth Riemannian metric.
We regard $\conn: \Sob{1,2}(\cM) \subset \Lp{2}(\cM) \to \Lp{2}(\cotanb\cM)$
to be the closed, densely-defined 
extension of the exterior derivative
on functions with domain $\Sob{1,2}(\cM)$, 
the first $\Lp{2}$-Sobolev space on $\cM$. 
Moreover, we let $\divv_\mg = -\adj{\conn}$, 
with domain $\dom(\divv_\mg) \subset \Lp{2}(\cotanb\cM)$.
Indeed, operator theory yields that
this is a densely-defined and closed operator
(see, for instance, Theorem 5.29 in \cite{Kato}
by Kato). 
The $\Lp{2}$-Laplacian on $(\cM,\mg)$
is then $\Lap_\mg = -\divv_\mg \conn$
which can easily be checked to be 
a non-negative self-adjoint operator with energy 
$ \sE[u] = \norm{\conn u}^2.$

In their paper \cite{BLM}, the authors
prove existence and uniqueness to
elliptic problems of the form 
\begin{equation*}
\tag{E}
\label{Def:E}
\Div_A u = -\divv_\mg A \conn u = f,
\end{equation*}
for suitable source data $f \in\Lp{2}(\cM)$,
where the coefficients
$A$ are symmetric, bounded, measurable
and for which there exists a $\kappa > 0$
satisfying $\inprod{Au,u} \geq  \kappa \norm{u}^2$.
The key to relating this equation 
to \eqref{Def:PE} is that, 
the source data $f$ can be chosen
independent of the coefficients $A$.

The operator $\Div_A$ is self-adjoint
on the domain $\dom(\Div_A)$ 
supplied via the Lax-Milgram theorem
by considering the symmetric form 
$J_A[u,v] = \inprod{A \conn u, \conn v}$
whenever $u, v \in \Sob{1,2}(\cM)$.
Since the coefficients are symmetric, 
we are able to write 
$J_A[u,v] = \inprod{\sqrt{\Div_A} u, \sqrt{\Div_A} v}.$
By the operator theory
of self-adjoint operators, 
we obtain that 
$\Lp{2}(\cM) = \nul(\Div_A) \oplus^\perp \close{\ran(\Div_A)}$, 
where by $\nul(\Div_A)$ and $\ran(\Div_A)$, 
we denote the \emph{null space} and \emph{range} of $\Div_A$
respectively.
Similarly, 
$\Lp{2}(\cM) = \nul(\sqrt{\Div_A}) \oplus^\perp \close{\ran(\sqrt{\Div_A})}$.
See, for instance,
the paper \cite{CDMcY} by Cowling, Doust, McIntosh and Yagi.

First, we note that, due to the divergence
structure of this equation,
an easy operator theory argument
yields $\nul(\Div_A) = \nul(\conn) = \nul(\sqrt{\Div_A})$. 
The characterisation of $\close{\ran(\Div_A)}$
independent of $\Div_A$
rests on the fact that,
by the compactness of $\cM$ and smoothness 
of  $\mg$, there exists a
\Poincare inequality of the form 
\begin{equation*}
\tag{P}
\label{Def:P}
\norm{ u - u_{\cM,\mg}}_{\Lp{2}} \leq C \norm{\conn u}_{\Lp{2}},
\end{equation*}
where $u_{\cM,\mg} = \fint_{\cM} u\ d\mu_\mg$
(see, for instance Theorem 2.10 in \cite{Hebey}  by 
Hebey).
The constant $C$ can be 
taken to be  $\lambda_1(\cM,\mg)$, 
the lowest non-zero eigenvalue of the
Laplacian $\Lap_\mg$ of $(\cM,\mg)$.
The space $\close{\ran(\Div_A)}$
and $\close{\ran(\sqrt{\Div_A})}$
can then be characterised as
the set 
$$ \ran = \set{ u \in \Lp{2}(\cM): \int_{\cM} u\ d\mu_\mg = 0}.$$
A proof of this can be found as
Proposition 4.1 in \cite{BLM}.

Recall that, again 
as a consequence of the fact that $(\cM,\mg)$
is smooth and compact, the embedding $E: \Sob{1,2}(\cM) \to \Lp{2}(\cM)$
is compact (see Theorem 2.9 in \cite{Hebey}).
In Proposition 4.4 in \cite{BLM},
the authors 
use this fact to show that the 
the spectrum of $\Div_A$
is \emph{discrete}, i.e., 
$\spec(\Div_A) = \set{0 = \lambda_0 \leq \lambda_1 \leq \dots \leq \lambda_k \leq \dots}.$
Coupled with the \Poincare inequality, 
we can obtain that the operator exhibits a spectral gap 
between the zero and the first-nonzero eigenvalues.
That is, $\lambda_0  < \lambda_1$.
Moreover, $\kappa \lambda_1(\cM,\mg) \leq \lambda_1$.

It is a fact from operator theory
that the operator $\Div_A$ 
preserves the
subspaces $\nul(\Div_A)$ and 
$\close{\ran(\Div_A)}$.
Consequently, the operator 
$\Div_A^R = \Div_A \rest{\close{\ran(\Div_A)}}$
has spectrum 
$\spec(\Div_A^R) = \set{0 <  \lambda_1 \leq \lambda_2 \leq \dots}.$
Collating these facts together,
we obtain the following.

\begin{theorem}
\label{Thm:EU}
For every $f \in \Lp{2}(\cM)$ satisfying $\int_{\cM} f\ d\mu_\mg = 0$,
we obtain a unique solution $u \in \dom(\Div_A) \subset \Sob{1,2}(\cM)$ 
with $\int_{\cM} u\ d\mu_\mg = 0$ to the equation
$\Div_A u = f$. This solution is given by
$u = (\Div_A^R)^{-1}f$.
\end{theorem}

For the purposes of legibility, we
write $\Div_A^{-1}$ in place of $(\Div_A^R)^{-1}$.
\section{An application to a geometric flow}
\label{Sec:App}

In this section, we describe an application 
of Theorem \ref{Thm:Main} to a geometric flow
first proposed by Gigli and Mantegazza in \cite{GM}.
In their paper, they consider solving the continuity 
equation
\begin{equation*}
\label{Def:GMC}
\tag{GMC}
-\divv_{\mg,y} \hk^\mg_t(x,y) \conn \phi_{t,x,v}(y) = \extd_{x}(\hk^\mg_t(x,y))(v),
\end{equation*}
for each fixed $x$,
where $\hk^\mg_t$ is the heat kernel 
of $\Lap_\mg$, $\divv_{\mg,y}$ denotes
the divergence operator acting on the variable $y$,
where $v \in \tanb_x \cM$, and $\extd_x(\hk^\mg_t(x,y)(v)$
is the directional derivative of $\hk^\mg_t(x,y)$
in the variable $x$ in the direction $v$.
They define a new family of metrics
evolving in time by the expression
\begin{equation*}
\label{Def:GM}
\tag{GM}
\mg_t(x)(u,v) = \int_{\cM} \mg(y)(\conn \phi_{t,x,u}(y), \conn \phi_{t,x,v}(y))\ \hk^\mg_t(x,y)\ d\mu_\mg(y).
\end{equation*}

As aforementioned, this flow is of importance
since it is tangential (a.e. along geodesics) 
to the Ricci flow when starting with a smooth 
initial metric.
Moreover, in \cite{GM}, the authors
demonstrate that this flow is equal to a certain
heat flow in the Wasserstein space, 
and define a flow of a distance
metric for the recently developed $\RCD$-spaces.
These are metric spaces that have a
notion of lower bound of a generalised Ricci 
curvature (formulated in the language of mass transport)
and for which their Sobolev spaces are Hilbert.
We refer the reader to the seminal work of  Ambrosio, Savar\'e,
and Gigli in \cite{AGS} as well as the work
of Gigli in \cite{G} for a
detailed description of these spaces and their properties.

In \cite{BLM}, the authors were interested in the
question of proving existence and regularity of
this flow when the metric $\mg$ was no longer assumed to be
smooth or even continuous.
The central geometric objects for them are
\emph{rough metrics}, which are a sufficiently large
class of symmetric tensor fields 
which are able to capture singularities, 
including, but not limited to, Lipschitz transforms
and certain conical singularities.
The underlying differentiable structure of the manifold
is always assumed to be smooth, and hence,
rough metrics capture \emph{geometric} singularities.

More precisely, let $\mgt$ 
be a symmetric $(2,0)$ tensor field
and suppose at each point $x \in \cM$, 
there exists a chart $(\psi_x, U_x)$ near
$x$ and a constant $C = C(U_x) \geq 1$
satisfying
$$ C^{-1} \modulus{u}_{\pullb{\psi_x}\delta(y)}
	\leq \modulus{u}_{\mgt(y)} \leq C \modulus{u}_{\pullb{\psi_x}\delta(y)},$$
for $y$ almost-everywhere (with respect to $\pullb{\psi_x}\Leb$,
the pullback of the Lebesgue measure) inside $U_x$,
where $u \in \tanb_y \cM$, and where $\pullb{\psi_x}\delta$ is the pullback 
of the Euclidean metric inside $(\psi_x, U_x)$.
A tensor field $\mgt$ satisfying this condition 
is called a \emph{rough metric}.
Such a metric may not, in general, induce 
a length structure, but (on a compact manifold) it will induce
an $n$-dimensional Radon measure.
 
Two rough metrics $\mgt_1$ and $\mgt_2$
are said to be $C$-close (for $C \geq 1$) if
$$ 
C^{-1} \modulus{u}_{\mgt_1(x)} 
	\leq\modulus{u}_{\mgt_2(x)}
	\leq C \modulus{u}_{\mgt_1(x)},$$
for almost-every $x$ and where $u \in \tanb_x \cM$.
For any two rough metrics, there exists
a symmetric measurable $(1,1)$-tensor field 
$B$ such that
$\mgt_1(Bu,v) = \mgt_2(u,v)$.
For $C$-close rough metrics,
$C^{-2} \modulus{u} \leq \modulus{B(x)u} \leq C^{2} \modulus{u}$
in either induced norm.
In particular, this means that 
their $\Lp{p}$-spaces are equal 
with equivalent norms. 
Moreover, Sobolev spaces exist as 
Hilbert spaces, and these spaces
are also equal with comparable norms.
On writing $\theta = \sqrt{\det B}$,
which denotes the density for
the change of measure $d\mu_{\mgt_2} = \sqrt{\det{B}}\ d\mu_{\mgt_1}$,
the divergence operators satisfy
$ \divv_{\mgt_2} = \theta^{-1} \divv_{\mgt_1}\theta B$,
and the Laplacian 
$\Lap_{\mgt_2} = \theta^{-1} \divv_{\mgt_1} \theta B \conn$.
Since we assume $\cM$ is compact, for any rough metric $\mgt$,
there exists a $C \geq 1$ and a smooth
metric $\mg$ that is $C$-close. 

As far as the author is
aware, the notion of a rough metric
was first introduced by the author
in his investigation 
of the geometric invariances
of the Kato square root problem
in \cite{BRough}. However, 
a notion close to this exists
in the work of Norris in \cite{Norris}
and the notion of $C$-closeness
between two continuous 
metrics can be found in \cite{Simon}
by Simon and in \cite{SC} by Saloff-Coste.

There is an important connection between divergence
form operators and rough metrics,
and this is crucial to  the analysis
carried out in \cite{BLM}.
The authors noticed that equation \eqref{Def:GMC}
and the flow \eqref{Def:GM}
still makes sense
if the initial metric $\mg$ was replaced
by a rough metric $\mgt$.
To fix ideas, let us denote a rough 
metric by $\mgt$ and by $\mg$, a smooth metric
that is $C$-close. 
In this situation, we can write the equation
\eqref{Def:GMC}  equivalently 
in the form
\begin{equation*}
\label{Def:GMC'}
\tag{GMC'}
-\divv_{\mg,y} \hk^\mgt_t(x,y) B \theta\conn \phi_{t,x,v} = \theta \extd_x(\hk^{\mgt}_t(x,y))(v).
\end{equation*}
Indeed, it is essential  
to  understand the heat
kernel of $\Lap_\mgt$ and
its regularity to 
make sense of the right hand side of this equation.
In \cite{BLM}, the authors assume 
$\hk^\mgt_t \in \Ck{0,1}(\cM)$ and
further 
assuming $\hk^\mg_t \in \Ck{k}(\cN^2)$,
for $k \geq 2$ and where $\emptyset \neq \cN \subset \cM$ 
represents a ``non-singular''
open set, they show the existence
of solutions to \eqref{Def:GMC'}
and provide a time evolving family 
of metrics $\mg_t$ defined via the equation \eqref{Def:GM}
on $\cN$ of regularity $\Ck{k-2,1}$.
We remark that this set typically arises
as $\cN = \cM \setminus \cS$
where $\cS$ is some singular part of $\mg$. 
For instance, 
for a cone attached to a sphere at the north 
pole, we have that $\cS = \set{p_{\text{north}}}$,
and on $\cN$, both the metric and heat 
kernel are smooth.

The aforementioned assumptions are not a restriction 
to the applications that the authors of \cite{BLM}
consider as their primarily goal was
to consider geometric conical singularities,
and spaces like a box in Euclidean space.
All these spaces are, in fact, $\RCD$-spaces and such spaces have been 
shown to always have Lipschitz heat kernels.
General rough metrics
may fail to be $\RCD$, and more seriously, 
even fail to induce a metric.
However, for such metrics, the following
still holds.

\begin{proposition}
For a rough metric $\mgt$, 
the heat kernel $\hk^\mgt_t$ for $\Lap_\mgt$ exists
and for every $t > 0$, there exists some $\alpha >0$ 
such that $\hk^\mgt_t \in \Ck{\alpha}(\cM)$.
\end{proposition}

This result is due to the fact that
the notion of \emph{measure contraction property}
is preserved under $C$-closeness,
and hence, by Theorem 7.4 in \cite{ST1}
by Sturm, 
one can obtain the
existence and regularity of the heat kernel 
by viewing $\Lap_\mgt$ 
as a divergence form operator on the nearby smooth metric $\mg$.
A more detailed proof of this fact can be 
found in the proof of Theorem 5.1 in \cite{BLM}.

In order to proceed, we note the following
existence and uniqueness
result to solutions of the equation \eqref{Def:GMC'}.
\begin{proposition}
\label{Prop:E} 
Suppose that $\hk^\mg_t \in \Ck{1}(\cN^2)$
where $\emptyset \neq \cN \subset \cM$
is an open set.
Then, for each $x \in \cN$,
the equation \eqref{Def:GMC'} has a unique
solution $\phi_{t,x,v} \in \Sob{1,2}(\cM)$ satisfying
$ \int_{\cM} \phi_{t,x,v}\ d\mu_\mgt = 0.$
This solution is given by
$$ \phi_{t,x,v} = \Div_{x}^{-1}(\theta \eta_{t,x,v}) - \fint_{\cM} \Div_{x}^{-1}(\theta \eta_{t,x,v})\ d\mu_\mgt,$$
where $\Div_{x} u = -\divv_{\mg,y} \hk^{\mg}_t(x,y)\conn u$
and $\eta_{t,x,v} = \extd_x(\hk^\mg_t(x,Y))(v)$.
\end{proposition}
\begin{proof}
We note that the proof of this
proposition runs in a very similar 
way to Proposition 4.6 and 4.7 in \cite{BLM}. 
Note that the first proposition simply requires
that $\hk^\mg_t \in \Ck{0}(\cM^2)$,
and that $\hk^\mg_t > 0$.
This latter inequality is yielded by 
Lemma 5.4 in \cite{BLM},
which again, only requires 
that $\hk^\mg_t \in \Ck{0}(\cM^2)$. 
\end{proof}

\begin{remark}
When inverting this operator $\Div_{x}$
as a divergence form operator on 
the nearby smooth metric $\mg$,
the solutions $\psi_{t,x,v} = \Div_{x}^{-1} (\theta \eta_{t,x,v})$
satisfy $\int_{\cM} \psi_{t,x,v}\ d\mu_\mg = 0$. 
The adjustment by subtracting
$\fint_{\cM} \psi_{t,x,v}\ d\mu_\mgt$
to this solution is to ensure
that $\int_{\cM} \phi_{t,x,v}\ d\mu_\mgt = 0$. 
That is, the integral with respect to $\mu_\mgt$,
rather than $\mu_\mg$, is zero.
\end{remark}

Collating these results together, and
invoking Theorem \ref{Thm:Main}, 
we obtain the following. 
\begin{theorem}
Let $\cM$ be a smooth, compact manifold,
and $\emptyset \neq \cN \subset \cM$, an
open set. Suppose that $\mgt$ is a rough metric
and that $\hk^\mgt_t \in \Ck{1}(\cN^2)$.
Then, $\mg_t$ as defined by \eqref{Def:GM}
exists on $\cN$ and it is continuous.
\end{theorem}
\begin{proof}
By Proposition \ref{Prop:E}, we obtain existence
of $\mg_t(x)$ for each $x \in \cN$
as a Riemannian metric. The fact that it
is a non-degenerate inner product
follows from similar argument
to that of the proof of Theorem 3.1  in \cite{BLM}, 
which only requires the continuity of $\hk^\mg_t$.

Now, to prove that $x \mapsto \mg_t(x)$
is continuous,it suffices to 
prove that  $x\mapsto \modulus{u}_{\mg_t(x)}^2$
as a consequence of polarisation.
Here, we fix a coordinate chart
$(\psi_x, U_x)$ near $x$ and consider
$u = \pushf{\psi_x^{-1}}\tilde{u}$, 
where $\tilde{u} \in \R^n$ is a constant vector
inside $(\psi_x, U_x)$.
In this situation, we note that \eqref{Def:GM}
can be written in the following way: 
$$ \modulus{u}_{\mg_t(x)}^2 
	= \inprod{\Div_x \phi_{t,x,u}, \phi_{t,x,u}}
	= \inprod{\eta_{t,x,u}, \phi_{t,x,u}}.$$
Now, to prove continuity, we need to prove that 
$\modulus{\modulus{u}_{\mg_t(x)} - \modulus{u}_{\mg_t(y)}}$
can be made small when $y$ is sufficiently close to
$x$. This is obtained if,
each of $\modulus{\inprod{\eta_{t,x,u} - \eta_{t,y,u}, \phi_{t,x,u}}}$
and $\modulus{\inprod{\eta_{t,y,u}, \phi_{t,x,u} - \phi_{t,y,u}}}$
can be made small.

The first quantity is easy: 
$$\modulus{\inprod{\eta_{t,x,u} - \eta_{t,y,u}, \phi_{t,x,u}}}
	\leq \norm{\eta_{t,x,u} - \eta_{t,y,u}} \norm{\phi_{t,x,u}},$$
and by our assumption on $\hk^\mg_t(x,z)$ that it is continuously
differentiable for $x \in \cN$ and $\Ck{\alpha}$
in $z$,
we have that  $(x,y) \mapsto \eta_{x,t,u}(y)$ is uniformly continuous 
on $K \times \cM$ for every $K \Subset \cN$ (open subset, compactly contained
in $\cN$) by the compactness of $\cM$. Thus, on fixing $K \Subset \cN$, 
we have that for $x, y \in K$,
$$\norm{\eta_{t,x,u} - \eta_{t,y,u}}
	\leq \mu_\mgt(\cM) \sup_{z \in \cM} \modulus{\eta_{t,x,u}(z) - \eta_{t,y,u}(z)}$$ 
and the right hand side
can be made small for $y$ sufficiently close to $x$.

Now, the remaining term can be
estimated in a similar way:
$$
\modulus{\inprod{\eta_{t,y,u}, \phi_{t,x,u} - \phi_{t,y,u}}}
	\leq \norm{\eta_{t,y,u}} \norm{\phi_{t,x,u} - \phi_{t,y,u}}.$$
First, observe that,
$\norm{\eta_{t,y,u}} = \norm{\eta_{t,y,u} - \eta_{t,x,u}} + \norm{\eta_{t,x,u}}$
and hence, by our previous argument, the first term can be made small 
and the second term only depends on $x$. Thus, it 
suffices to prove that
$\norm{\phi_{t,x,u} - \phi_{t,y,u}}$ can be made small.
Note then that, 
\begin{align*}
\norm{\phi_{t,x,u} - \phi_{t,y,u}} 
	&\leq \norm{\Div_x^{-1} \theta \eta_{t,x,u} - \Div_x^{-1}\theta \eta_{t,y,u}}\\
		&\qquad\qquad+ \mu_\mgt(\cM) \cbrac{\fint_{\cM} \Div_x^{-1} \theta \eta_{t,x,u} - \Div_x^{-1}\theta \eta_{t,y,u}\ d\mu_\mgt} \\
	&\leq (1 + \mu_\mgt(\cM)) \norm{\Div_x^{-1} \theta \eta_{t,x,u} - \Div_x^{-1}\theta \eta_{t,y,u}},
\end{align*}
where the last inequality follows from the Cauchy-Schwarz
inequality applied to the average.

Again, by the assumptions on $\hk^\mgt_t$,
$$\norm{\B\theta\hk^\mgt_t(x,\mdot) - B\theta\hk^\mgt_t(y,\mdot)}_\infty
	\lesssim \norm{B\theta}_\infty \sup_{z \in \cM} \modulus{\hk^\mgt_t(x,z) - \hk^\mgt_t(y,z)}$$
which shows that  $x \mapsto B(\mdot)\theta(\mdot) \hk^\mgt_t(x,\mdot)$
is $\Lp{\infty}$-continuous.
Moreover, we have already shown 
that $(w,z) \mapsto \eta_{t,x,u}(z)$
is uniformly continuous on $K \times \cM$ 
for $K \Subset \cN$ and hence,
since $\theta$ is essentially
bounded from above and below, $x \mapsto \theta \eta_{t,x,u}$ 
is $\Lp{2}$-continuous.
Thus, we apply Theorem \ref{Thm:Main}
to obtain the conclusion.
\end{proof}

\begin{remark}
If we assume that $\mgt$ is a rough metric
on $\cM$, but away from some
singular piece $\cS$, we assume that the metric
is $\Ck{1}$, then, by the results in \S6 
of \cite{BLM}, we are able to
obtain that the heat kernel $\hk^\mgt_t \in \Ck{2}(\cM \setminus \cS)$.
Hence, we can apply this theorem
to obtain that the flow is continuous on $\cM \setminus \cS$.
In \cite{BLM} a similar theorem is 
obtained (Theorem 3.2) but requires
the additional assumption that $\hk^\mgt_t \in \Ck{1}(\cM^2)$.
\end{remark}
\section{Proof of the theorem}
\label{Sec:Kato}

In this section, we prove the main theorem
by first proving a
homogeneous Kato square root result. We begin
with a description of functional calculus
tools required phrase and resolve the problem.

\subsection{Functional calculi for sectorial operators}

Let $\Hil$ be a \emph{complex} Hilbert space 
and $T: \dom(T) \subset \Hil \to \Hil$
a linear operator. Recall that the 
\emph{resolvent set} of $T$ denoted by 
$\rset(T)$ consists of 
$\zeta \in \C$ such that $(\zeta\iden - T)$ 
has dense range and a bounded inverse
on its range. It is easy 
to see that $(\zeta \iden - T)^{-1}$ 
extends uniquely to bounded operator on the whole space.
The \emph{spectrum} is then 
$\spec(T) = \C \setminus \rset(T)$.

Fix $\omega \in [0, \pi/2)$ and define 
the $\omega$-bisector and open $\omega$-bisector
respectively as 
\begin{align*} 
&\Sec{\omega} = \set{ \zeta \in \C: \modulus{\arg \zeta} \leq \omega
	\ \text{or}\ \modulus{\arg(-\zeta)} \leq \omega\ \text{or}\ \zeta = 0}\ \text{and} \\
&\OSec{\omega} = \set{ \zeta \in \C: \modulus{\arg \zeta} < \omega
	\ \text{or}\ \modulus{\arg(-\zeta)} < \omega\,\ \text{and}\  \zeta \neq 0}.
\end{align*}
An operator $T$ is said to be $\omega$-\emph{bi-sectorial} if
it is closed, $\spec(T) \subset \Sec{\omega}$, and
whenever $\mu \in (\omega, \pi/2)$,
there exist $C_{\mu}$ satisfying
the \emph{resolvent bounds}:
$\modulus{\zeta} \norm{(\zeta\iden - T)^{-1}} \leq C_\mu$
for all $\zeta \in \OSec{\mu} \setminus \Sec{\omega}$.
Bi-sectorial operators naturally generalise
self-adjoint operators: a self-adjoint operator
is $0$-bi-sectorial.
Moreover, bi-sectorial operators
admit a spectral decomposition
of the space
$\Hil = \nul(T) \oplus \close{\ran(T)}$. 
This sum is not, in general, 
orthogonal, but it is always
topological. By $\proj_{\nul(T)}: \Hil \to \nul(T)$
we denote the continuous
projection from $\Hil$ to $\nul(T)$ 
that is zero on $\close{\ran(T)}$.
 
Fix some $\mu \in (\omega, \pi/2)$
and by $\Psi(\OSec{\mu})$ denote the class
of holomorphic functions $\psi: \OSec{\mu} \to \C$
for which there exists an $\alpha > 0$ 
satisfying 
$$ \modulus{\psi(\zeta)} \lesssim 
	\frac{\modulus{\zeta}^\alpha}{1 + \modulus{\zeta}^{2\alpha}}.$$
For an $\omega$-bi-sectorial operator $T$,
we define a bounded
operator $\psi(T)$ via
$$ 
\psi(T)u = \frac{1}{2\pi\imath} \oint_{\gamma} \psi(\zeta)(\zeta\imath - T)^{-1}u\ d\zeta,$$
where $\gamma$ is an unbounded contour enveloping
$\Sec{\omega}$ counter-clockwise inside $\OSec{\mu}$ and
where the integral is defined via Riemann sums.
The resolvent bounds for the operator $T$ 
coupled with the decay of the function $\psi$
yields the absolute convergence of this integral.

Now, suppose there exists a $C > 0$
so that $\norm{\psi(T)} \leq C \norm{\psi}_\infty$. 
In this situation, we say that 
$T$ has a \emph{bounded functional calculus}.
Let $\Hol^\infty(\OSec{\mu})$
be the class of bounded functions 
$f: \OSec{\mu} \union \set{0} \to \C$
for which $f\rest{\OSec{\mu}}: \OSec{\mu} \to \C$
is holomorphic. 
For such a function, there is always
a sequence of functions $\psi_n \in \Psi(\OSec{\mu})$
which converges to $f \rest{\OSec{\mu}}$
in the compact-open topology. 
Moreover, if $T$ has a bounded 
functional calculus, the 
limit $\lim_{n \to \infty} \psi_n(T)$ 
exists in the strong operator topology, and
hence, we define 
$$ f(T)u = f(0) \proj_{\nul(T)}u + \lim_{n \to \infty} \psi_n(T)u.$$
The operator $f(T)$ is independent of the sequence
$\psi_n$, it is bounded, and moreover, 
satisfies $\norm{f(T)} \leq C \norm{f}_\infty$.
By considering the function $\chi^{+}$, 
which takes the value  $1$
for $\re \zeta > 0$ and $0$ otherwise, 
and $\chi^{-}$ taking $1$ for $\re \zeta < 0$
and $0$ otherwise, we define $\sgn = \chi^{+} - \chi^{-}$.
It is readily checked that $\sgn \in \Hol^{\infty}(\OSec{\mu})$
for any $\mu$ and hence, for $T$
with a bounded functional calculus,
the $\chi^{\pm}(T)$ define projections.
In addition to the spectral decomposition,
we obtain
$\Hil = \nul(T) \oplus \ran(\chi^{+}(T)) \oplus 
\ran(\chi^{-}(T))$.

Lastly, we remark that
a quantitative criterion for demonstrating that
$T$ has a bounded functional calculus
is to find $\psi \in \Psi(\OSec{\mu})$
satisfying the \emph{quadratic estimate}
$$ \int_{0}^\infty \norm{\psi(tT)u}^2\ \frac{dt}{t} \simeq \norm{u}^2,\quad u \in \close{\ran(T)}.$$ 
In particular, this criterion
facilitates the use of harmonic 
analysis techniques to prove
the boundedness of the functional calculus.
We refer the reader to \cite{ADMc} 
by Albrecht, Duong and McIntosh for a more complete
treatment of these ideas.

\subsection{Homogeneous Kato square root problem}

We have already given a brief historical 
overview of the Kato square root problem 
in the introduction. An important advancement, from the point of view of
proving such results on manifolds,
was the development of the first-order Dirac-type operator
approach by Axelsson, Keith and McIntosh 
in \cite{AKMc}.
Their set of hypotheses (H1)-(H8)
is easily accessed in the literature,
and therefore, we shall omit repeating them here.
For the benefit of the 
reader, we remark that the particular form that we use
here is listed in \cite{BMc}.

Let $\Hil = \Lp{2}(\cM) \oplus \Lp{2}(\cotanb\cM)$
and 
$$\Gamma = \begin{pmatrix} 0 & 0 \\ \conn & 0\end{pmatrix},
\quad\text{and}\quad
\adj{\Gamma} = \begin{pmatrix} 0 & -\divv \\ 0 & 0 \end{pmatrix}.$$
Then, for elliptic (possibly complex and non-symmetric)
coefficients $B \in \Lp{\infty}(\Tensors[1,1]\cM)$,
satisfying $\re\inprod{Bu, u} \geq \kappa_1 \norm{u}^2$,
and $b \in \Lp{\infty}(\cM)$ with $\re b(x) \geq \kappa_2$,
define
$$B_1 = \begin{pmatrix} b & 0 \\ 0 & 0\end{pmatrix},
\quad\text{and}\quad
B_2 = \begin{pmatrix} 0 & 0 \\ 0 & B \end{pmatrix}.$$
Define the Dirac-type operators
$\Pi_B = \Gamma + B_1 \adj{\Gamma} B_2$
and $\Pi = \Gamma + \adj{\Gamma}$.
The first operator is bi-sectorial 
and the second is self-adjoint (but with spectrum possibly 
on the whole real line).

First, we note that by bi-sectoriality, 
$$\Hil = \dom(\Pi) \oplus^\perp \close{\ran(\Pi)} 
	= \dom(\Pi_B) \oplus \close{\ran(\Pi_B)},$$ 
where the second direct sum is topological but
not necessarily orthogonal.
In particular, the first 
direct sum yields that
$\Lp{2}(\cM) = \nul(\conn) \oplus^\perp \close{\ran(\divv)}$
and $\Lp{2}(\cotanb \cM) = \nul(\divv) 	\oplus^\perp \close{\ran(\conn)}$.
We observe the following. 

\begin{lemma}
The space $\close{\ran(\divv)} = \set{u \in \Lp{2}(\cM): \int_{\cM} u = 0}$.
\end{lemma}
\begin{proof}
Let $u \in \close{\ran(\divv)}$. Then,  there
is a sequence $u_n \in \ran(\divv)$ such that 
$u_n \to u$. Indeed, $u_n = \divv v_n$, for some
vector field $v_n \in \dom(\divv)$. Thus, 
$$\int_{\cM} u\ d\mu_{\mgt} 
	=  \int_{\cM} \lim_{n\to \infty} \divv v_n\ d\mu_{\mgt}
	= \lim_{n \to \infty} \inprod{\divv v_n, 1} = 0.$$

Now, suppose that $\int_{\cM} u\ d\mu_\mgt = 0$.
Then, since $(\cM,\mgt)$ admits a \Poincare
inequality, we have that $\inprod{u, v} = 0$
for all $v \in \nul(\conn)$.
But since we have that 
$\Lp{2}(\cM) = \nul(\conn) \oplus^\perp \close{\ran(\divv)}$
via spectral theory,
we obtain that $u \in \close{\ran(\divv)}$.
\end{proof}

With this lemma, we obtain the following
coercivity estimate.

\begin{lemma}
\label{Lem:Cor}
Let $u \in \ran(\Pi) \intersect \dom(\Pi)$. 
Then, there exists a constant $C > 0$ such that
$\norm{u} \leq C \norm{ \Pi u}$.
\end{lemma}
\begin{proof}
Fix $u = (u_1, u_2) = \ran(\Pi) = \ran(\divv) \oplus^\perp \ran(\conn)$.
Then, $\norm{\Pi u} = \norm{\conn u_1} + \norm{\divv u_2}.$
By the \Poincare inequality along with the previous 
lemma, we obtain that $\norm{\conn{u_1}} \geq C_1 \norm{u_1}$.
For the other term, note that $\divv  u_2 = \divv \conn v = \Lap v$
for some $v \in \dom(\conn)$. Thus,
$$\norm{\Lap v} = \norm{\sqrt{\Lap} \sqrt{\Lap} v}
	\geq C_1 \norm{\sqrt{\Lap} v} = C_1 \norm{\conn v} = C_1 \norm{u_2}.$$
On setting $C = C_1$, we obtain the conclusion.
\end{proof}

Indeed, this is the key ingredient 
to obtain a bounded
functional  calculus for
the operator $\Pi_B$.

\begin{theorem}[Homogenous Kato square root problem for compact manifolds]
\label{Thm:Kato}
On a compact manifold $\cM$ with a smooth metric
$\mg$, 
the operator $\Pi_B$ admits a bounded 
functional calculus. In particular,
$\dom(\sqrt{b \divv B\conn}) = \Sob{1,2}(\cM)$
and $\norm{ \sqrt{b \divv B\conn} u} \simeq \norm{\conn u}$.
Moreover, whenever $\norm{\tilde{b}}_\infty < \eta_1$
and $\norm{\tilde{B}}_{\infty} < \eta_2$, 
where $\eta_i < \kappa_i$,
we have the following Lipschitz estimate
$$ \norm{\sqrt{b \divv B \conn} u - \sqrt{ (b + \tilde{b}) \divv(B + \tilde{B})\conn} u}
	\lesssim (\norm{\tilde{b}}_\infty + \norm{\tilde{B}}_\infty) \norm{ \conn u}$$
whenever $u \in \Sob{1,2}(\cM)$. The implicit constant depends on $B_i$ and $\eta_i$. 
\end{theorem}
\begin{proof}
Our goal is to check the Axelsson-Keith-McIntosh
hypotheses (H1)-(H8) as listed
in \cite{BMc} to invoke Theorem 4.2 and obtain 
a bounded functional calculus for $\Pi_B$.

To avoid unnecessary repetition by listing this framework,
we leave it to the reader to check \cite{BMc}.
However, for completeness of the proof,
we will remark on why the bulk of these
hypothesis are automatically true.

First, by virtue of the fact that we are on 
a smooth manifold with a smooth metric, 
we have that $\modulus{\Ric}  \lesssim 1$,
and $\inj(\cM, \mg) > \kappa > 0$.
Coupling this with the fact that $\Gamma$
is a first-order differential operator
makes their hypotheses (H1)-(H7) and (H8)-1
valid immediately.
The hypotheses (H1)-(H6) are valid
as a consequence of their Theorem 6.4 and
Corollary 6.5 in \cite{BMc}.
The proof of (H7) is contained in their
Theorem 6.2, as is the proof of
(H8)-1, which follows by bootstrapping
the \Poincare inequality \eqref{Def:P}
and coupling this with their Proposition 5.3.

It only remains to prove their (H8)-2: 
that there exists a $C > 0$
such that $\norm{\conn  u} + \norm{u} \leq C \norm{\Pi u}$,
whenever $u \in \ran(\Pi) \intersect \dom(\Pi)$.

Fix such a $u = (u_1, u_2)$
and note that $u_1 = \divv v_2$
for some $v_2 \in \dom(\divv)$
and $u_2 = \conn v_1$ for
some $v_1 \in \dom(\conn)$.
Then,
$$\norm{\conn u}^2 
	= \norm{\conn u_1}^2 + \norm{\conn u_2}^2
	= \norm{\conn\divv v_2}^2 + \norm {\conn^2 v_1}^2.$$
Also, 
$$\norm{\Pi v}^2 = \norm{\divv \conn v_1}^2 + \norm{\conn \divv v_2}^2.$$ 
Thus, it suffices to 
estimate the term $\norm{\conn^2 v_1}$
above from $\norm{\Lap v_1}$. 
By exploiting the
fact that $\Ck[c]{\infty}$
functions are dense in both $\dom(\Lap)$
and $\Sob{2,2}(\cM)$ on a compact
manifold, 
the Bochner-Weitzenb\"ock identity
yields
$\norm{\conn^2 v_1}^2 \lesssim 
	\norm{\Lap v_1 }^2 +  \norm{ v_1}^2$.
Now, $u_2 = \conn v_1 \in \ran(\conn)$
and we can assume that $u_2 \neq 0$.
Thus, $v_1 \not \in \nul(\conn)$
and hence, $\int_{\cM} v_1\ d\mu_\mgt = 0$.
Thus, by invoking the \Poincare inequality, 
we obtain that
$\norm{v_1} \leq C \norm{\conn v_1} = \norm{u_2}$.
On combining these estimates, we obtain 
that 
$\norm{\conn u} \lesssim \norm{\Pi u}$.
In Lemma \ref{Lem:Cor}, we have already
proven that $\norm{u} \lesssim \norm{\Pi u}$.

This allows us to invoke Theorem 4.2 in \cite{BMc}, 
which says that the operator $\Pi_B$
has a bounded functional calculus.
The first estimate in the conclusion is then immediate.

For the Lipschitz estimate, 
by the fact that that $\Pi_B$ has a bounded
functional calculus, we can apply
Corollary 4.6 in \cite{BMc}. This
result states that for multiplication
operators $A_i$ satisfying
satisfying
\begin{enumerate}[(i)]
\item $\norm{A_i}_\infty \leq \eta_i$,
\item $A_1 A_2 \ran(\Gamma), B_1 A_2 \ran(\Gamma),
	A_1 B_2 \ran(\Gamma) \subset \nul(\Gamma)$, and 
\item $A_2 A_1 \ran(\adj{\Gamma}), B_2 A_1 \ran(\adj{\Gamma}),
	A_2 B_1 \ran(\adj{\Gamma}) \subset \nul(\adj{\Gamma})$,
\end{enumerate}
we obtain 
that for an appropriately chosen $\mu < \pi/2$, 
and for all $f \in \Hol^\infty(S_{\mu}^o)$,
$$\norm{f(\Pi_B) - f(\Pi_{B+A})}
	\lesssim (\norm{A_1}_\infty + \norm{A_2}_\infty)
	\norm{f}_\infty.$$

Setting 
$$A_1 = \begin{pmatrix} \tilde{b} & 0 \\ 0 & 0 \end{pmatrix},
\quad\text{and}\quad
A_2 = \begin{pmatrix} 0 & 0 \\ 0 & \tilde{B} \end{pmatrix},$$
it is easy to see that these conditions are satisfied,
and by repeating the argument in 
Theorem 7.2 in \cite{BMc}
for our operator $\Pi_B$, 
we obtain the  Lipschitz estimate in the conclusion.
\end{proof}

\subsection{The main theorem}

Let us now return to the
proof of Theorem \ref{Thm:Main}.
Recall the operator 
$\Div_{x}u = -\divv A_x \conn u$, 
and that $\inprod{A_x u, u} \geq \kappa_x \norm{u}^2$, 
for $u \in \Lp{2}(\cotanb\cM)$.

A direct consequence of the 
Kato square root result
from our previous sub-section is
then the following.

\begin{corollary}
\label{Cor:Part}
Fix $x \in \cM$ and $u \in \Sob{1,2}(\cM)$. 
If 
$\norm{A_x - A_y} \leq \zeta < \kappa_x$, 
then for $u \in \Sob{1,2}(\cM)$, 
$$ 
\norm{\sqrt{\Div_{x}} u - \sqrt{\Div_{y}} u} \lesssim \norm{\A_x - \A_y}_\infty \norm{\conn u}.$$
The implicit constant depends on $\zeta$ and $A_x$.
\end{corollary}

In turn, this implies the
following.

\begin{corollary}
\label{Cor:Diff}
Fix $x \in \cM$ and 
suppose that $\norm{A_x - A_y} \leq \zeta < \kappa_x$.
Then, 
$$ \norm{\Div_x^{-1}\eta_x - \Div_y^{-1}\eta_y} \lesssim \norm{\A_x - \A_y}_\infty \norm{\eta_x} + \norm{\eta_x - \eta_y},$$
whenever $\eta_x, \eta_y \in \Lp{2}(\cM)$ satisfies $\int_{\cM} \eta_x\ d\mu_\mg  = \int_{\cM} \eta_y\ d\mu_\mg = 0$.
The implicit constant depends on  $\zeta$, $\kappa_x$, and $A_x$.
\end{corollary}
\begin{proof}
First consider the operator 
$T_x = \sqrt{\Div_x}$, and fix $u \in \Lp{2}(\cM)$ such that
$\int_{\cM} u\ d\mu_\mg = 0$. We prove that
$\norm{T_x^{-1}u - T_y^{-1}u} \leq \norm{\A_x - \A_y}_\infty \norm{u}.$

Observe that $\dom(T_x) = \Sob{1,2}(\cM)$
and so $T_x^{-1}u = T_x^{-1}(T_y T_y^{-1})u =  (T_x^{-1}T_y) T_y^{-1}u$
since $T_y^{-1} u \in \Sob{1,2}(\cM)$.
Also, since $T_x^{-1}T_x = T_x T_x^{-1}$ on $\Sob{1,2}(\cM)$,
we have that $T_y^{-1} u = T_x^{-1} T_x L^{-1}_y u$. 
Thus,
\begin{multline*}
\norm{T_x^{-1} u - T_y^{-1} u}
	= \norm{T_x^{-1} T_y T_y^{-1} u - T_x^{-1} T_x T_y^{-1}u} 
	= \norm{T_x^{-1}(T_y - T_x)T_y^{-1} u} \\
	\lesssim \norm{(T_y - T_x)T_y^{-1}u}
	\lesssim \norm{\A_x - \A_y}_\infty \norm{\conn T_y^{-1} u},
\end{multline*} 
where the penultimate inequality 
follows from Corollary \ref{Cor:Part}.

On letting $J_x[u] = \inprod{\A_x \conn u, \conn u} \geq \kappa_x \norm{\conn u}^2$,
we note that, for $\norm{\conn u} \neq 0$, 
$$\kappa_x - \kappa_y \leq \frac{ J_x[u] - J_y[u]}{\norm{\conn u}^2}\leq \norm{\A_x - \A_y}_\infty \leq \zeta < \kappa_x.$$
This gives us that $\kappa_x - \zeta \geq \kappa_y$ and 
$\kappa_x - \zeta > 0$ by our hypothesis, and hence, 
$$(\kappa_x - \zeta) \norm{\conn u}^2 \leq \kappa_y \norm{\conn u}^2 \leq J_y[u] = \norm{T_y u}^2.$$
Thus, $\norm{\conn T_y^{-1} u} \leq (\kappa_x - \zeta)^{-1} \norm{u}$, and hence,
$$\norm{T_x^{-1} u - T_y^{-1} u}
	\lesssim \norm{\A_x - \A_y}_\infty \norm{u},$$
where the implicit constant depends on $\zeta$, $\kappa_x$ and $A_x$.

Next, let $v_x, v_y \in \Lp{2}(\cM)$ satisfy $\int_{\cM} v_x\ d\mu_\mg  = \int_{\cM} v_y\ d\mu_\mg = 0$
and note that
\begin{align*}
\norm{T_x^{-1}v_x - T_y^{-1}v_y} 
	&\leq \norm{T_x^{-1}v_x - T_y^{-1}v_x} + \norm{T_y^{-1}(v_x - v_y)}\\
	&\lesssim \norm{\A_x - \A_y}_\infty \norm{v_x} + \norm{(T_x^{-1} - T_y^{-1})(v_x - v_y)}
						+ \norm{T_x^{-1}(v_x - v_y)} \\
	&\lesssim \norm{\A_x - \A_y}_\infty \norm{v_x} + \norm{\A_x - \A_y})_\infty \norm{v_x - v_y}
				+ \norm{v_x - v_y} \\
	&\lesssim \norm{\A_x - \A_y}_\infty \norm{v_x} + \norm{v_x - v_y},
\end{align*}
where the constant depends on  $\zeta$, $\kappa_x$, and $A_x$.
Now, putting $v_x = \Div_x^{-\frac{1}{2}}\eta_x = T_x^{-1}\eta_x$, 
and similarly choosing $v_y$,
since we assume  $\int_{\cM} \eta_x\ d\mu_\mg = \int_{\cM} \eta_y\ d\mu_\mg = 0$,
the same is satisfied for $v_x$ and $v_y$.
Hence, 
we apply what we have just proved to 
obtain 
\begin{align*}
\norm{\Div_x^{-1} \eta_x - \Div_y^{-1}\eta_y}
	&\lesssim \norm{\A_x - \A_y}_\infty \norm{\Div_{x}^{-\frac{1}{2}}\eta_x}
		+ \norm{T_x^{-1}\eta_x - T_y^{-1} \eta_y} \\
	&\lesssim \norm{\A_x - \A_y}_\infty \norm{\eta_x} 
	+ \norm{\A_x - \A_y}_\infty \norm{\eta_x}
		+ \norm{\eta_x - \eta_y}	\\
	&\lesssim \norm{\A_x - \A_y}_\infty \norm{\eta_x} + \norm{\eta_x - \eta_y}.
\end{align*}
This proves the claim.
\end{proof}

With the aid of this,
the proof of Theorem \ref{Thm:Main} is
immediate. 

\begin{proof}[Proof of Theorem \ref{Thm:Main}]
Fix $\epsilon \in (0, \kappa_x)$
and by the assumption that $x \mapsto \eta_x$ is $\Lp{2}$-continuous
at $x$ and that $x \mapsto A_x$ is $\Lp{\infty}$-continuous
at $x$, we have a $\delta = \delta(x,\epsilon)$
such that
$$\norm{\eta_x - \eta_x} < \epsilon  
\quad \text{and}\quad
\norm{A_x - A_y}_\infty < \epsilon.$$
Thus, in invoking Corollary \ref{Cor:Diff}, we obtain 
$\norm{u_x - u_y} \lesssim \epsilon$
where the implicit constant only depends on $x$.
\end{proof}

\providecommand{\bysame}{\leavevmode\hbox to3em{\hrulefill}\thinspace}
\providecommand{\MR}{\relax\ifhmode\unskip\space\fi MR }
\providecommand{\MRhref}[2]{%
  \href{http://www.ams.org/mathscinet-getitem?mr=#1}{#2}
}
\providecommand{\href}[2]{#2}

\setlength{\parskip}{0mm}

\end{document}